\documentclass[12pt,fleqn,dvips]{article}
\usepackage{amsmath,amsfonts,amssymb,psfrag,graphicx}%,chicago,verbatim}

\newenvironment{proof}[1][Proof]{\noindent\textbf{#1} }{\ \rule{0.5em}{0.5em}}
\newtheorem{proposition}{Proposition}

\begin{document}

\begin{center}
{\LARGE  Information Percolation}\\

\bigskip

%\centerline{\large Preliminary draft}

\bigskip

\centerline{Darrell Duffie}
\centerline{Stanford University}

\medskip

\centerline{Gaston Giroux} \centerline{unaffiliated}

\medskip

\centerline{Gustavo Manso}
\centerline{Massachusetts Institute of Technology}

\bigskip

\centerline{\today}
\end{center}

\section{Introduction}

For a setting in which a large number of asymmetrically informed
agents are randomly matched into groups over time, exchanging their
information with each other when matched, we provide an explicit
solution for the dynamics of the cross-sectional distribution of
posterior beliefs. We also show that convergence of the cross-sectional 
distribution of beliefs to a common posterior is exponential and
that the rate of convergence does not depend on the size of the
groups of agents that meet. The rate of convergence is merely the mean
rate at which an individual agent is matched.

For example, suppose that each agent has $\lambda$ meetings per
year, in expectation. At each meeting, say an auction, $m-1$ other
agents are randomly selected to attend. Each agent at the meeting
reveals to the others a summary statistic of his or her posterior,
such as a bid for an asset, reflecting the agent's originally
endowed information and any information learned from prior meetings.
Over time, the conditional beliefs held across the population of
agents regarding a variable of common concern (such as the payoff of
the auctioned asset) converge to a common posterior. We construct an
associated mathematical model of information transmission and
calculate explicitly the cross-sectional distribution of the
posterior beliefs held by the agents at each time. We show that 
convergence of these posteriors to a common posterior is exponential
at the rate of $\lambda$, regardless of the number $m$ of agents at 
each meeting.

An important role of markets and organizations, as argued for
example by Hayek (1945) and Arrow (1974), is to facilitate the
transmission of information that is dispersedly held by its
participants. Our results suggest that varying the size of the
groups in which individuals exchange information does not facilitate
information transmission, at least in terms of the rate of
convergence of posteriors. This point is further addressed at the
end of the paper.

Previous studies have considered the problem of information
aggregation in various contexts. For example, Grossman (1981) 
proposes the concept of rational-expectations equilibrium to 
capture the idea that prices aggregate information that is
disperse in the economy.
Wilson (1977), Milgrom (1981), Vives (1993), Pesendorfer and 
Swinkels (1997), and Reny and Perry (2006) provide strategic 
foundations for the rational-expectations equilibrium concept 
in centralized markets.
In a number of important settings, however, agents learn
from local interactions. For example, in over-the-counter
markets, agents learn from the bids of other agents in 
privately held auctions. Wolinsky (1990) and Blouin and 
Serrano (2002) study information percolation in
these markets.%
\footnote{Rubinstein and Wolinsky (1985) and Gale (1986a, 1986b) 
study decentralized markets without asymmetric information. 
Satterthwaite and Shneyerov (2003) study decentralized markets with
private-value asymmetric information.} 
In social learning settings, 
agents learn from direct interactions
with other agents. Banerjee and Fudenberg (2004) study
information percolation in a social learning context. 
In contrast to previous studies of learning through local 
interactions, we allow for meetings that have more 
than two agents, and we explicitly characterize the 
percolation of information and provide rates of convergence 
of the cross-sectional distribution of 
beliefs to a common posterior.

Our results extend those of Duffie and Manso (2007), who provided an
explicit formula for the Fourier transform of the cross-sectional
distribution of posterior beliefs in the same setting, but did not
offer an explicit solution for the distribution itself, and did not
characterize the rate of convergence of the distribution.

Section \ref{sec:basic} provides the model setting. Section
\ref{sec:two} provides our results for the traditional search-market
setting of bilateral $(m=2)$ contacts. This also serves as an
introduction to the results for the case of general $m$, which are
presented in Section \ref{sec:m}.

\section{The Basic Model}
\label{sec:basic}

The model of information percolation is that of Duffie and Manso
(2007). A probability space $(\Omega, {\cal F}, {\rm P})$ and a
``continuum'' (a non-atomic finite measure space $(G, {\cal
G},\gamma)$) of agents are fixed. Without loss of generality, the
total quantity $\gamma(G)$ of agents is 1. A random variable $X$ of
potential concern to all agents has two possible outcomes, $H$
(``high'') and $L$ (``low''), with respective probabilities $\nu$
and $1-\nu$.

Each agent is  initially endowed with a sequence of signals that may
be informative about $X$. The signals $\{s_1,\ldots,s_n\}$
primitively observed by a particular agent are, conditional on $X$,
independent with outcomes 0 and 1 (Bernoulli trials). The number
$n\geq 0$ of signals as well as the probability distributions of the
signals may vary across agents. Without loss of generality, we
suppose that $\operatorname{P}(s_i = 1 \,|\, H) \geq
\operatorname{P}(s_i = 1 \,|\, L)$.  A signal $i$ is
\textit{informative} if $\operatorname{P}(s_i = 1 \,|\, H) >
\operatorname{P}(s_i = 1 \,|\, L)$. For any pair of agents, their
sets of originally endowed signals are independent.

By Bayes' rule, conditional on signals $\{s_1,\ldots,s_n\}$, the
posterior probability that $X$ has a high outcome is
\begin{equation}
\operatorname{P}(X=H \,|\, s_1,\ldots,s_n) = \left[ 1+
\frac{1-\nu}{\nu} \left(\frac{1}{2}\right)^{\theta} \right]^{-1},
\label{eq:beliefs}
\end{equation}
where the ``type'' $\theta$ of this set of signals is
\begin{equation}
\theta = \sum_{i=1}^{n} \left(s_i \operatorname{log}_{1/2}
\frac{\operatorname{P}(s_i = 1 \,|\, L)}%
{\operatorname{P}(s_i = 1 \,|\, H)} +
(1-s_i)\operatorname{log}_{1/2}\frac{1-\operatorname{P}(s_i = 1
\,|\, L)} {1-\operatorname{P}(s_i = 1 \,|\, H)}\right).
\label{eq:types}
\end{equation}
The higher the type $\theta$ of the set of signals, the higher is
the posterior probability that $X$ is high.

Any particular agent is matched to other agents at each of a
sequence of Poisson arrival times with a mean arrival rate
(intensity) $\lambda$, which is common across agents. At each
meeting time, $m-1$ other agents are randomly selected from the
population of agents.\footnote{That is, each of the $m-1$ matched
agents is chosen at random from the population, without replacement,
with the uniform distribution, which we can take to be the
agent-space measure $\gamma.$  Duffie and Sun (2007) provide a
complete construction for independent random matching from a large
set (a non-atomic measure space) of agents, for the case $m=2$.} The
meeting group size $m$ is a parameter of the information model that
we shall vary. We assume that, for almost every pair of agents, the
matching times and counterparties of one agent are independent of
those of the other. We do not show the existence of such a random
matching process.\footnote{For the case of groups of size $m=2$,
Duffie and Sun (2007) show existence for the discrete-time analogue
of this random matching model. For the case of a finite number of
agent types, the associated exact law of large numbers for the
cross-sectional distribution of the type processes is provided by
Duffie and Sun (2005). Giroux (2005) proves convergence of the
naturally associated finite-agent discrete-time model to the
analogous continuous-time model matching model
of Duffie, G\^arleanu, and Pedersen (2005) as the number of agents grows
large.}

Suppose that whenever agents meet they communicate to each other
their posterior probabilities, given all information to the point of
that encounter, of the event that $X$ is high. Duffie and Manso
(2007) provide an example of a market setting in which this
revelation of beliefs occurs through the observation of bids
submitted by risk-neutral investors in an auction for a forward
contract on an asset whose payoff is $X$. From Proposition $3$ in
Duffie and Manso (2007), whenever an agent of type $\theta$ meets an
agent with type $\phi$ and they communicate to each other their
posterior distributions of $X$, they both attain the posterior type
$\theta+\phi$. The same proof implies that whenever $m$ agents of
respective types $\phi_1,\ldots,\phi_m$ share their beliefs, they
attain the common posterior type $\phi_1+\cdots+\phi_m$.

We let $\mu_t$ denote the cross-sectional distribution of posterior
types in the population at time $t$. That is, for any real interval
$(a,b)$, $\mu_t((a,b))$ (also denoted $\mu_t(a,b)$ for simplicity)
is the fraction of the population whose type at time $t$ is in
$(a,b)$. Because the total quantity $\gamma(G)$ of agents is 1, we
can view $\mu_t$ as a probability distribution. The initial
distribution $\mu_0$ of types is that induced by some particular
initial allocation of signals to agents. In the following analysis
we assume that there is a positive mass of agents that has at
least one informative signal. This implies that the first
moment $m_1(\mu_0)$ is strictly positive if $X=H$, and that
$m_1(\mu_0)<0$ if $X=L$. We assume that the initial law $\mu_0$ has
a moment generating function, $z\mapsto \int e^{z\theta}\,
\mu_0(d\theta)$, that is finite on a neighborhood of $z=0$.

\section{Two-Agent Meetings}
\label{sec:two}

We now calculate the explicit belief distribution in the population
at any given time, and the rate of convergence of beliefs to a
common posterior, in a setting with $m=2$  agents at each meetings.
This is the standard setting for search-based models of labor,
money, and asset markets. In this setting, the cross-sectional
distribution of types is determined by the evolution equation
\begin{equation}
\mu_t = \mu_0 + \lambda \int_0^t (\mu_s*\mu_s-\mu_s)\, ds,
\label{eq:evolution}
\end{equation}
where $*$ is the convolution operator. This is intuitively
understood if $\mu_t$ has a density $f_t(\,\cdot\,)$, in which case
the density $f_t(\theta)$ of agents of type $\theta$ is reduced at
the rate $\lambda f_t(\theta)$ at which agents of type $\theta$ meet
other agents and change type, and is increased at the aggregate rate
$\lambda \int f_t(\theta-y)f_t(y)\, dy$ at which an agent of some
type $y$ meets an agent of type $\theta -y$, and therefore becomes
an agent of type $\theta$.

The following result provides an explicit solution for the
cross-sectional type distribution, in the form of a Wild
summation.\footnote{See Wild (1951).}
\begin{proposition}
The unique solution of (\ref{eq:evolution}) is
\begin{equation}
\mu_t = \sum\limits_{n\geq 1} e^{- \lambda t} (1-e^{-\lambda
t})^{n-1} \mu_0^{*n}, \label{eq:wild}
\end{equation}
where $\nu^{*n}$ denotes the $n$-fold convolution of a measure
$\nu$.
 \label{prop:wild}
\end{proposition}

\begin{proof}
As in Duffie and Manso (2007), we write the evolution equation
(\ref{eq:evolution}) in terms of the Fourier transform $
\varphi(\,\cdot\,,t)$ of $\mu_t$, as
\begin{equation}
\frac{\partial \varphi(s,t)}{\partial t}=-\lambda \varphi(s,t) +
\lambda \varphi^2(s,t),
\end{equation}
with solution
\begin{equation}
\varphi(s,t) = \frac{\varphi(s,0)}{e^{\lambda t}(1-\varphi(s,0)) + \varphi(s,0)}.
\end{equation}
This solution can be expanded as
\begin{equation}
\varphi(s,t) =  \sum\limits_{n\geq1} e^{- \lambda t}(1-e^{- \lambda
  t})^{n-1} \varphi^n(s,0),
\end{equation}
which is identical to the Fourier transform of the right-hand side
of (\ref{eq:wild}).
\end{proof}

\bigskip

This solution for the cross-sectional distribution of types is
converted to an explicit distribution for the cross-sectional
distribution $\pi_t$ of posterior probabilities that $X=H$,
 using the fact that
\begin{equation} \pi_t(0,b)=\mu_t\left(-\infty,\, 
\log_{1/2}\frac{(1-b)\nu}{(1-\nu)b}\right). \label{conversion}
\end{equation}
Like $\mu_t$, the beliefs distribution $\pi_t$ has an outcome that
differs depending on whether $X=H$ or $X=L$.

We now provide explicit rates of convergence of the
cross-sectional distribution of beliefs to a common posterior. In
our setting, it turns out that all agents' beliefs converge to that
of complete information, in that any agent's posterior probability
of the event $\{X=H\}$ converges to 1 on this event and to zero
otherwise. In general, we say that $\pi_t$ converges to a common
posterior distribution $\pi_\infty$ if, almost surely, $\pi_t$
converges in distribution to $\pi_\infty$, and we say that convergence
is exponential at the rate $\alpha>0$ if there are constants $\kappa_0$ and
$\kappa_1$ such that, for any $b$ in $(0,1),$
$$e^{-\alpha t}\kappa_0\leq \vert \pi_t(0,b)-\pi_\infty(0,b)\vert
\leq e^{-\alpha t}\kappa_1.$$ Thus, if there is a rate of
convergence, it is unique.

\begin{proposition}
Convergence of the cross-sectional distribution of beliefs 
to that of complete information is exponential
at the rate $\lambda$.
\end{proposition}

\begin{proof}
Because of (\ref{conversion}), the rate of convergence of $\pi_t$ is
the same as the rate of convergence to zero or 1, for any $a$, of
$\mu_t(-\infty,a)$. We will provide the rate of convergence to zero
of $\mu_t(-\infty,a)$ on the event $X=H$. A like argument gives the
same rate of convergence to 1 on the event $X=L$.

From ($\ref{eq:wild}$),
\begin{equation}
\mu_t(-\infty,a) \geq e^{-\lambda t} \mu_0(-\infty,a).
\label{eq:lower}
\end{equation}
We fix $n_0$ such that $m_1(\mu_0)>a/n$ for $n>n_0$ and we let
$\lbrace Y_n \rbrace_{n\geq 1}$ be independent random variables
with distribution $\mu_0$. Then,
\begin{eqnarray}
\mu_t(-\infty,a) & = & \sum\limits_{n = 1}^{n_0} e^{-\lambda t} (1-e^{-\lambda t})^{n-1}
\operatorname{P}\left[\sum_{i=1}^n\left(Y_i-\frac{a}{n}\right) \leq 0 \right] \nonumber \\
                 &      & + \sum\limits_{n \,= \, n_0+1}^{\infty} e^{-\lambda t} (1-e^{-\lambda t})^{n-1}
\operatorname{P}\left[\sum_{i=1}^n\left(Y_i-\frac{a}{n}\right) \leq 0 \right].
\label{eq:upper}
\end{eqnarray}
It is clear that there exists a constant $\beta$ such that, for all
$t$, the first term on the right-hand side of equation
($\ref{eq:upper}$) is less than $\beta e^{-\lambda t}$. Therefore,
we only need to worry about the second term on the right-hand side
of equation ($\ref{eq:upper}$). From a standard result in
probability theory,\footnote{See, for example, Rosenthal (2000), pp.
90-92.} if $Y$ is a random variable with a finite strictly positive
mean and a moment generating function that is finite on $(-c,0]$ for
some $c>0$, then
$\operatorname{P}(Y \leq 0) \leq \inf\limits_{-c<s<0} E[e^{sY}] < 1$.
For $n>n_0$, for some fixed $c>0$, we then have
\begin{equation}
\operatorname{P}\left[\sum_{i=1}^n\left(Y_i-\frac{a}{n}\right) \leq 0 \right] \leq
\left(\inf\limits_{-c<s<0} E\left[e^{s(Y_1-a/n)}\right]\right)^n
\leq  e^{ac} \gamma^n, \label{eq:gamma}
\end{equation}
where $\gamma < 1$. From ($\ref{eq:gamma}$), we conclude that the
second term on the right-hand side of equation ($\ref{eq:upper}$) is
bounded by $e^{ac} \frac{\gamma}{1-\gamma} e^{-\lambda t}$.
Therefore,
\begin{equation}
\mu_0(-\infty,a) e^{-\lambda t} \leq \mu_t(-\infty,a) \leq \left(\beta + e^{a c}
\frac{\gamma}{1-\gamma}\right) e^{-\lambda t},
\end{equation}
and the proof is complete.
\end{proof}

\section{Multi-Agent Meetings}
\label{sec:m}

For the case of $m$ agents at each meeting, the evolution of the
cross-sectional distribution of types is similarly given by:
\begin{equation}
\mu_t = \mu_0 + \lambda \int_0^t (\mu_s^{*m}-\mu_s)\, ds,
\label{eq:evolutionm}
\end{equation}
as explained in Duffie and Manso (2007). A solution for the
cross-sectional distribution of beliefs at any time $t$ is given
explicitly by (\ref{conversion}) and the following extension of the
Wild summation formula for the type distribution.

\begin{proposition}
The unique solution of (\ref{eq:evolutionm}) is
\begin{equation}
\mu_t = \sum\limits_{n\geq1} a_{[(m-1)(n-1)+1]} e^{-\lambda t}
(1-e^{-(m-1) \lambda t})^{n-1}\mu_0^{*[(m-1)(n-1)+1]},
\label{eq:wildm}
\end{equation}
where $a_1=1$ and, for $n>1$,
\begin{equation}
a_{(m-1)(n-1)+1} = \frac{1}{m-1}\left(1-
   \sum\limits_{\left\lbrace  i_1,\ldots,i_{(m-1)}<n \atop \sum i_{k}=n+m-2 \right\rbrace}
    \prod\limits_{k=1}\limits^{m-1} a_{[(m-1)(i_k-1)+1]} \right).
\end{equation}
\label{prop:wildm}
\end{proposition}

\begin{proof}
From (\ref{eq:evolutionm}), the Fourier transform of $\mu_t$
satisfies
\begin{equation}
\frac{\partial \varphi(s,t)}{\partial t}=-\lambda \varphi(s,t) +
\lambda \varphi^m(s,t),
\end{equation}
whose solution satisfies
\begin{equation}
\varphi(s,t)^{m-1}=\frac{\varphi(s,0)^{m-1}}{e^{(m-1)\lambda t}(1-\varphi^{m-1}(s,0))+\varphi^{m-1}(s,0)}.
\end{equation}
Following steps analogous to those of Proposition $\ref{prop:wild}$,
\begin{equation}
\mu_t^{*(m-1)} = \sum\limits_{n\geq1} e^{-(m-1) \lambda t} (1-e^{-(m-1)
  \lambda t})^{n-1} \mu_0^{*(m-1)n}.
\label{eq:wildmconv}
\end{equation}

Let $\nu_t$ denote the right-hand side of ($\ref{eq:wildm}$). By
recursively calculating the convolution,
\begin{eqnarray}
\nu_t^{*(m-1)} & = & \left(\sum\limits_{n\geq1} a_{[(m-1)(n-1)+1]}
e^{-\lambda t}
(1-e^{-(m-1) \lambda t})^{n-1}\mu_0^{*[(m-1)(n-1)+1]}\right)^{*(m-1)} \nonumber\\
             & = & \sum\limits_{n\geq1} \beta_n e^{-(m-1) \lambda t} (1-e^{-(m-1)
             \lambda t})^{n-1}\mu_0^{*(m-1)n}  \\
             & = & \sum\limits_{n\geq1} e^{-(m-1) \lambda t} (1-e^{-(m-1)
 \lambda t})^{n-1} \mu_0^{*(m-1)n},
\end{eqnarray}
where
\[\beta_n=\sum\limits_{\left\lbrace  i_1,\ldots,i_{(m-1)} \atop
                   \sum i_{k}=n+m-2 \right\rbrace}
              \prod\limits_{k=1}\limits^{m-1}
               a_{[(m-1)(i_k-1)+1]},\]
and where  the last equality follows from the definition of $
a_{[(m-1)(n-1)+1]}$ for $n\geq 1$. Thus,
$\nu_t^{*(m-1)}=\mu_t^{*(m-1)},$ and it remains to show that the
distribution $\mu_t$ is uniquely characterized by its convolution of
order $m-1$. This follows\footnote{Because, on $\{X=H\}$, the
derivative of the moment generating function of $\mu_0$ at zero is
the first moment of $\mu_0$, which is positive, the moment
generating function is strictly less than 1 in an interval
$(-\epsilon,0]$, for a sufficiently small $\epsilon>0$. This implies
that there is an analogous explicit solution for the moment
generating function of $\mu_t^{*n}$, for any $n$ and $t$, on a small
negative interval. The $(m-1)$-st root of the moment generating
function of $\mu_t^{*(m-1)}$, on such an interval, uniquely
determines the associated measure $\mu_t$. For additional details,
see Billingsley (1986), p. 408.} from the fact that $\mu_0$, and
therefore $\mu_t^{*k}$ for any $t$ and $k$, have a moment generating
function in a neighborhood of zero and a non-zero first moment on
the event $\{X=H\}$.
\end{proof}

\begin{proposition}
For any meeting group size $m$,  convergence of the cross-sectional 
distribution of beliefs to that of complete information 
is exponential at the rate $\lambda$.
\end{proposition}

\proof Again, it is enough to derive the rate of convergence of
$\mu_t(-\infty,a)$ to zero on the event $\{X=H\}$. From
($\ref{eq:wildm}$),
\begin{equation}
\mu_t(-\infty,a) \geq e^{-\lambda t} \mu_0(-\infty,a).
\label{eq:upper2}
\end{equation}
Now, from ($\ref{eq:wildmconv}$) and our analysis in Section
$\ref{sec:two}$, we know that for some constant $\kappa>0,$
\begin{equation}
\mu^{*(m-1)}(-\infty,(m-1)a) \leq \kappa e^{-(m-1) \lambda t}.
\end{equation}
From the fact that
\begin{equation}
(\mu(-\infty,a))^{m-1} \leq \mu^{*(m-1)}(-\infty,(m-1)a),
\end{equation}
we conclude that
\begin{equation}
\mu(-\infty,a) \leq \kappa^{\frac{1}{m-1}} e^{- \lambda t}.
\label{eq:lower2}
\end{equation}
From ($\ref{eq:upper2}$) and ($\ref{eq:lower2}$), it follows that the
rate of convergence of $\mu_t(-\infty,a)$ to zero is $\lambda$,
completing the proof.
\endproof

\bigskip
Because the expected rate at which a particular individual enters
meetings is $\lambda$ per year, independence and a formal
application of the law of large numbers implies that the total
quantity of $m$-agent meetings per year is $\lambda/m$, almost
surely. So the total annual attendance at meetings is almost surely
$\lambda$ per year, invariant%
\footnote{This is not about large numbers, or uncertainty.
For example, suppose each member of a group $\{A,B,C,D\}$ of 4
agents holds one meeting with a different member of the group.
For example, $A$ meets with $B$, and $C$ meets with $D$. Then
there is a total of two meetings, and each individual attends
one meeting. If the 4 agents meet all together, once,
we have the same total attendance, and the same rate
at which each individual attends a meeting.}
to $m$.
Our results show that total attendance at meetings is what
matters for information convergence rates.

We have not shown that our invariance result extends from the case
of a constant group size to a model in which the group size varies
at random from meeting to meeting, say with the same mean group size
across meetings. This is in fact the case under technical conditions
on the distribution of group sizes, as related to us by Semyon Malamud 
in a subsequent private communication.

\bigskip

\quad

\bigskip
%\newpage

\bigskip\noindent
{\bf References}

\bigskip\noindent
Arrow, Kenneth. 1974. {\it The Limits of Organization}. Norton, New York.

\bigskip\noindent
Banerjee, Abhijit, and Drew Fudenberg. 2004. ``Word-of-Mouth Learning.''
{\it Games and Economic Behavior}, {\bf 46:} 1-22.

\bigskip\noindent
Billingsley, Patrick. 1986. {\it Probability and Measure, Second
Edition}, Wiley, New York.

\bigskip\noindent
Duffie, Darrell, Nicolae G\^ arleanu, and Lasse Heje Pedersen. 2005.
``Over-the-Counter Markets.'' {\it Econometrica}, {\bf 73:}
1815-1847.

\bigskip\noindent
Duffie, Darrell, and Gustavo Manso. 2007. ``Information Percolation in
Large Markets.''
{\it American Economic Review Papers and Proceedings}, {\bf 97:} 203-209.

\bigskip\noindent
Duffie, Darrell,  and Yeneng Sun. 2007. ``Existence of Independent
Random Matching.'' {\it Annals of Applied Probability}, {\bf 17:}
386-419.

\bigskip\noindent
Duffie, Darrell, and Yeneng Sun. 2005. ``The Exact Law of Large
Numbers for Independent Random Matching.'' Stanford University
Working Paper.

\bigskip\noindent
Gale, Douglas. 1986a. ``Bargaining and Competition Part I: Characterization.'' {\it Econometrica}, {\bf 54:} 785-806.

\bigskip\noindent
Gale, Douglas. 1986b. ``Bargaining and Competition Part II: Existence.'' {\it Econometrica}, {\bf 54:} 807-818.

\bigskip\noindent
Giroux, Gaston. 2005. ``Markets of a Large Number of Interacting
Agents.'' Working Paper (available from the author upon request).

\bigskip\noindent
Hayek, Friedrich. 1945. ``The Use of Knowledge in Society.'' {\it American Economic Review}, {\bf 35:} 51\
9-530.

\bigskip\noindent
Malamud, Semyon. 2007. ``Private Communication of November 19, 2007,''  Department of Mathematics, ETH, Zurich.

\bigskip\noindent
Milgrom, Paul. 1981. ``Rational Expectations, Information Acquisition, and Competitive Bidding.''
{\it Econometrica}, {\bf 50:} 1089-1122. 

\bigskip\noindent
Pesendorfer, Wolfgang and Jeroen Swinkels. 1997. ``The Loser's Curse and Information Aggregation in 
Common Value Auctions.'' {\it Econometrica}, {\bf 65:} 1247-1281.

\bigskip\noindent
Reny, Philip and Motty Perry. 2006. ``Toward a Strategic Foundation for Rational Expectations Equilibrium.''
{\it Econometrica}, {\bf 74:} 1231-1269.

\bigskip\noindent
Rosenthal, Jeffrey. 2000. {\it A First Look at Rigorous Probability
Theory}. World Scientific. River Edge, New Jersey.

\bigskip\noindent
Rubinstein, Ariel and Asher Wolinsky. 1985. ``Equilibrium in a Market With Sequential Bargaining.'' {\it Econometrica}, 
{\bf 53:} 1133-1150.

\bigskip\noindent
Satterthwaite, Mark, and Artyom Shneyerov. 2007. ``Dynamic Matching,
Two-sided Incomplete Information, and Participation Costs: Existence
and Convergence to Perfect Competition.'' {\it Econometrica}
{\bf 75:} 155-200.

\bigskip\noindent
Vives, Xavier. 1993. ``How Fast do Rational Agents Learn.''
{\it Review of Economic Studies}, {\bf 60:}
329-347.

\bigskip\noindent
Wild, E.  1951. ``On Boltzmann's Equation in the Kinetic Theory of
Gases,'' {\it Proceedings of the Cambridge Philosophical Society},
{\bf 47:} 602-609.

\bigskip\noindent
Wilson, Robert. 1977. ``Incentive Efficiency of Double Auctions,''
{\it The Review of Economic Studies}, {\bf 44:} 511-518.

\end{document}